\documentclass[10pt]{article}
\usepackage[utf8]{inputenc}
\usepackage[T1]{fontenc}
\usepackage{lmodern}
\usepackage[french,english]{babel}
\usepackage{amsmath}
\usepackage{amsthm}
\usepackage{amssymb}
\usepackage{mathrsfs}
\usepackage{geometry}
\usepackage{graphicx}
\usepackage{comment}
\usepackage{url}
\usepackage{titlesec}
\usepackage[all]{xy}
\usepackage{enumitem}
\usepackage{amsfonts,amsxtra}
\usepackage{stmaryrd}
\usepackage{tikz}
\usepackage[colorlinks=true]{hyperref}

\newtheorem{df}{Definition}[section]
\newtheorem{rem}[df]{Remark}
\newtheorem{ex}[df]{Example}
\newtheorem{thm}[df]{Theorem}
\newtheorem{pp}[df]{Proposition}
\newtheorem{lm}[df]{Lemma}

\newtheorem{cor}[df]{Corollary}

\newtheorem{thm intro}{Theorem}

\def\su{\mathfrak{su}}

\def\sl{\mathfrak{sl}}

\def\ppp{\mathfrak{p}}

\def\ss{\mathfrak{s}}
\def\ll{\mathfrak{l}}

\def\l{\lambda}

\newcommand*\quot[2]{{^{\textstyle #1}\big/_{\textstyle #2}}}
\def\square{\quot{k^*}{k^{*2}}}

\def\ssquare{\quot{k^*}{k^*_{-1}}}
\def\PP{\mathfrak{P}}

\title{\Large{\textbf{Involutions of sl(2,k) and non-split, three-dimensional simple Lie algebras}}}
\author{Philippe Meyer\\
\normalsize	Mathematical Institute, University of
Oxford, Oxford, Oxfordshire OX2 6GG, UK\\
\normalsize \texttt{philippe.meyer@maths.ox.ac.uk}}
\date{\vspace{-5ex}}

\begin{document}
\maketitle
{\let\thefootnote\relax}
{\let\thefootnote\relax}

\begin{center}
\textbf{Abstract}
\end{center}
We give a process to construct non-split, three-dimensional simple Lie algebras from involutions of $\sl(2,k)$, where $k$ is a field of characteristic not two. Up to equivalence, non-split three-dimensional simple Lie algebras obtained in this way are parametrised by a subgroup of the Brauer group of $k$ and are characterised by the fact that their Killing form represents $-2$. Over local and global fields we re-express this condition in terms of Hilbert and Legendre Symbols and give examples of three-dimensional simple Lie algebras which can and cannot be obtained by this construction over the field of rationals.
\vspace{0.4cm}

\begin{comment}
\noindent
\textbf{Key words}: three-dimensional simple Lie algebra $\cdot$ quaternion algebra $\cdot$ Cartan involution.

\noindent
\textbf{2020 Mathematics Subject Classification}: 11E08, 11R52, 17B05.
\end{comment}
\section{Introduction}

It is well-known that the non-split three-dimensional simple real Lie algebra $\su(2)$ can be constructed from $\sl(2,\mathbb{R})$ equipped with a Cartan involution. In this article, we generalise this process to fields $k$ of characteristic not two and obtain non-split three-dimensional simple Lie algebras from $\sl(2,k)$ equipped with a certain type of involution.\\

\vspace{-0.2cm}
To this end, in Section \ref{Section 3DLA construction}, we start from an involution $\sigma$ of $\sl(2,k)$ such that the linear map ${\rm ad}(x)$ is not diagonalisable for all fixed points $x$ of $\sigma$. We show in Theorem \ref{thm split pair -> non-split pair} that the Lie algebra obtained by our construction is non-split if and only if $K(x,x)$ is not a sum of two squares in $k$ for all fixed points $x$ of $\sigma$, where $K$ is the Killing form of $\sl(2,k)$. Three-dimensional simple Lie algebras which can be obtained by this construction are characterised by the fact that their Killing form is anisotropic and represents $-2$ (Proposition \ref{pp obtainable}).
\vspace{0.2cm}

In Section \ref{section generalities} we recall how to associate a quaternion algebra to a three-dimensional simple Lie algebra. Quaternion algebras do not define a subgroup of the Brauer group $B(k)$, however, the set of quaternion algebras associated to the non-split three-dimensional simple Lie algebras constructed from involutions of $\sl(2,k)$ does (Proposition \ref{pp subgroup of the Brauer group} and Theorem \ref{thm corresponces Brauer}). Finally in Section \ref{section local and global fields} we characterise which non-split three-dimensional simple Lie algebras are obtainable over local and global fields in terms of Hilbert and Legendre Symbols (Propositions \ref{critere local field} and \ref{thm global field}) and we give explicit rational three-dimensional simple Lie algebras which can and cannot be obtained by this construction in Example \ref{examples over Q}.

\section*{Acknowledgements}

The author wants to express his gratitude to Marcus J. Slupinski for his inspiring discussions, his encouragement and his expert advice.
\vspace{0.5cm}

\textit{Throughout this paper, the field $k$ is always of characteristic not two.}

\section{Generalities about three-dimensional simple Lie algebras} \label{section generalities}

In this section, we recall some results about three-dimensional simple Lie algebras (see \cite{Jacobson1958} or \cite{MALCOLMSON1992} for details) and describe their correspondence with quaternion algebras.

\begin{df}
Let $\lbrace x,y,z\rbrace$ be the canonical basis of $k^3$ and let $\alpha,\beta \in k^*$. Define an antisymmetric bilinear bracket $[\phantom{v},\phantom{v}] : k^3\times k^3\rightarrow k^3$ by 
$$[x,y]:=z, \qquad [y,z]:=\alpha x, \qquad [z,x]:=\beta y$$
and denote the algebra $(k^3,[\phantom{v},\phantom{v}])$ by $L(\alpha,\beta)$.
\end{df}

The algebra $L(\alpha,\beta)$ is a three-dimensional simple Lie algebra, its Killing form is $<-2\beta,-2\alpha,-2\alpha\beta>$ and we have the following result:

\begin{pp}
If $\ss$ is a three-dimensional simple Lie algebra, then there exist $\alpha,\beta \in k^*$ such that $\ss$ is isomorphic to $L(\alpha,\beta)$.
\end{pp}

\begin{rem}
A ternary quadratic form $q$ is isometric to the Killing form of a three-dimensional simple Lie algebra if and only if $disc(q)=[-2]\in\square$.
\end{rem}

The Lie algebra $\sl(2,k)$ is isomorphic to $L(-1,1)$. We say that a three-dimensional simple Lie algebra $\ss$ is split if it is isomorphic to $\sl(2,k)$. If there exists a non-zero $h\in \ss$ such that ${\rm ad}(h)$ is diagonalisable then $\ss$ is split.

\begin{pp} \label{pp poly car and isotropic}
Let $\ss$ be a three-dimensional simple Lie algebra, $K$ be its Killing form and $h \in \ss$. The characteristic polynomial of ${\rm ad}(h)$ is $-X(X^2-\frac{K(h,h)}{2})$. In particular, ${\rm ad}(h)$ is diagonalisable if and only if $\frac{K(h,h)}{2}$ is a non-zero square in $k$.
\end{pp}

\begin{proof}
Straightforward calculation.
\end{proof}

It is well-known that the imaginary part of a quaternion algebra $\mathcal{H}$ is a three-dimensional simple Lie algebra for the bracket defined by the commutator. Let $\alpha, \beta \in k^*$. Recall (\cite{Vigneras80},\cite{LAM}) that the quaternion algebra $\bigl(\begin{smallmatrix}
\alpha, \beta\\ \hline k
\end{smallmatrix} \bigr)$ is the $k$-algebra on two generators $i,j$ with the defining relations:
\begin{equation*}
i^2=\alpha, \qquad j^2=\beta, \qquad ij=-ji.
\end{equation*}
We note that the element $ij$ verifies $(ij)^2=-\alpha \beta$. Furthermore, $\lbrace 1,i,j,ij \rbrace$ form a $k$-basis for $\bigl(\begin{smallmatrix}
\alpha, \beta\\ \hline k
\end{smallmatrix} \bigr)$ and $\bigl(\begin{smallmatrix}
\alpha, \beta\\ \hline k
\end{smallmatrix} \bigr)$ is a central simple unital four-dimensional associative, non-commutative composition algebra for the norm form
$$N(a+bi+cj+dij)=(a+bi+cj+dij)\overline{(a+bi+cj+dij)}=a^2-\alpha b^2-\beta c^2+\alpha \beta d^2,$$
where $\overline{(a+bi+cj+dij)}:=a-(bi+cj+dij)$ for all $a+bi+cj+dij\in \bigl(\begin{smallmatrix}
\alpha, \beta\\ \hline k
\end{smallmatrix} \bigr)$.
\vspace{0.2cm}

The imaginary part of the quaternion algebra $\bigl(\begin{smallmatrix}
\alpha, \beta\\ \hline k
\end{smallmatrix} \bigr)$ together with the commutator is isomorphic to the Lie algebra $L(-\beta,-\alpha)$ since $Im\bigl(\begin{smallmatrix}
\alpha, \beta\\ \hline k
\end{smallmatrix} \bigr)={\rm Span}<i,j,ij>$ and 
$$[\frac{i}{2},\frac{j}{2}]=\frac{ij}{2}, \qquad [\frac{j}{2},\frac{ij}{2}]=-\beta \cdot \frac{i}{2}, \qquad [\frac{ij}{2},\frac{i}{2}]=-\alpha \cdot \frac{j}{2}.$$
Conversely, from a three-dimensional simple Lie algebra $\ss$ we can reconstruct a quaternion algebra as follows.
\begin{df}
Let $\ss$ be a three-dimensional simple Lie algebra, $K$ be its Killing form and $\mathcal{H}(\ss)$ be the vector space defined by $\mathcal{H}(\ss):=k\oplus \ss$. Define the product $~\cdot~ : \mathcal{H}(\ss) \times \mathcal{H}(\ss) \rightarrow \mathcal{H}(\ss)$ by:
\begin{enumerate}[label=\alph*)]
\item for $a,b \in k$, $a\cdot b:=ab$ (the field product on $k$) ;
\item for $a \in k$ and $v \in \ss$, $a\cdot v:=v\cdot a:=av$ (the scalar multiplication of $k$ on $\ss$) ;
\item for $v,w \in \ss$, 
$$v\cdot w:=\frac{K(v,w)}{8}\cdot 1+ \frac{[v,w]}{2}.$$
\end{enumerate}
\end{df}

\noindent
We define a norm form on $\mathcal{H}(\ss)$ by 
$$N(x):=x\cdot\overline{x}=(a^2-\frac{K(v,v)}{8})\cdot 1 \qquad \forall x=a\cdot 1+v \in \mathcal{H}(\ss),$$
where $\overline{x}:=a\cdot 1-v$.

\begin{pp} \label{pp quaternion from 3DSLA}
Let $\ss$ be a three-dimensional simple Lie algebra. The vector space $\mathcal{H}(\ss)=k\oplus\ss$ with the product $\cdot$ above is a quaternion algebra. Furthermore, if $\ss\cong L(\alpha,\beta)$ then $\mathcal{H}(\ss)\cong \bigl(\begin{smallmatrix}
-\beta, -\alpha\\ \hline k
\end{smallmatrix} \bigr)$.
\end{pp}

\begin{proof}
Straightforward calculation using the identity
$$K([v,w],[v,w])=\frac{1}{2} \left( K(v,w)^2-K(v,v)K(w,w)\right) \qquad \forall v,w\in \ss.$$
\end{proof}

\begin{ex}
\begin{enumerate}[label=\alph*)]
\item Over $\mathbb{R}$, the Lie algebra $\su(2)$ is the imaginary part of the classical quaternion algebra $\mathbb{H}$.
\item The Lie algebra $\sl(2,k)$ is the imaginary part of the split quaternion algebra $ M_2(k)\cong\bigl(\begin{smallmatrix}
1, -1\\ \hline k
\end{smallmatrix} \bigr)$.
\end{enumerate}
\end{ex}

The following is a résumé of the correspondence between three-dimensional simple Lie algebras, their Killing forms and their associated quaternion algebras.

\begin{pp} \label{equivalences quaternion alg}
For $\alpha,\beta,\alpha',\beta' \in k^*$ the following are equivalent:
\begin{enumerate}[label=\alph*)]
\item the Lie algebras $L(\alpha,\beta)$ and $L(\alpha',\beta')$ are isomorphic ;
\item the quaternion algebras $\bigl(\begin{smallmatrix}
-\alpha, -\beta\\ \hline k
\end{smallmatrix} \bigr)$ and $\bigl(\begin{smallmatrix}
-\alpha', -\beta'\\ \hline k
\end{smallmatrix} \bigr)$ are isomorphic ;
\item the quadratic forms $<\beta,\alpha,\alpha\beta>$ and $<\beta',\alpha',\alpha'\beta'>$ are isometric.
\end{enumerate}
\end{pp}

\begin{cor} \label{cor 3DSLA iso iff Killing iso}
Two three-dimensional simple Lie algebras are isomorphic if and only if their Killing forms are isometric. In particular, a three-dimensional simple Lie algebra is split if and only if its Killing form is isotropic.
\end{cor}

\section{Construction of non-split three-dimensional simple Lie algebras from involutions of sl(2,k)} \label{Section 3DLA construction}

In this section we give a construction of non-split three-dimensional simple Lie algebras from involutions of $\mathfrak{sl}(2,k)$ and characterise those which one can be obtained in this way. We show that non-split three-dimensional simple Lie algebras constructed from these involutions define a particular subgroup of the Brauer group $B(k)$ of $k$. We first introduce some notation. Let
$$k^*_{-1}:=\lbrace x^2+y^2 ~ | ~ x,y \in k \rbrace \setminus \lbrace 0 \rbrace.$$
This is a subgroup of $k^*$ and if $-1 \in k^{*2}$, then $\ssquare \cong \lbrace 1 \rbrace$ since
$$(\frac{1+\Delta}{2})^2+(\sqrt{-1}\Big(\frac{1-\Delta}{2}\Big))^2=\Delta \qquad \forall \Delta \in k^*.$$

Let $\ss$ be a split three-dimensional simple Lie algebra, $K$ be its Killing form and $\sigma$ be a non-trivial involutive automorphism of $\ss$ such that
$$[\frac{K(x,x)}{2}]\neq 1 \in \square \qquad \forall x \in \ss^{\sigma}.$$
To this data, we are going to associate another three-dimensional simple Lie algebra $\ss'$. Let $\Lambda \in k^*$ be such that
$$[\Lambda]=[\frac{K(x,x)}{2}] \in \square \qquad \forall x \in \ss^{\sigma},$$
and $\l$ be a square root of $\Lambda$ in a non-trivial quadratic extension of $k$. Since $\sigma$ is involutive we have
$$\ss\cong \ll \oplus \ppp,$$
where $\ll$ is the one-dimensional eigenspace for the eigenvalue $1$ and $\ppp$ is the two-dimensional eigenspace for the eigenvalue $-1$. Let $\ss'$ be the three-dimensional simple $k$-Lie algebra
$$\ss':=\ll \oplus \l \ppp$$
with the Lie bracket extended from $\ss$:
$$[a+\l b,c+\l d]=([a,c]+\Lambda[b,d])+\l([b,c]+[a,d]) \qquad \forall a+\l b,c+\l d \in \ss'.$$

\begin{thm} \label{thm split pair -> non-split pair}
Let $\ss$ be a split three-dimensional simple Lie algebra, $K$ be its Killing form and $\sigma$ be a non-trivial involutive automorphism of $\ss$ such that
$$[\frac{K(x,x)}{2}]\neq 1 \in \square \qquad \forall x \in \ss^{\sigma}.$$
The three-dimensional simple Lie algebra $\ss'$ associated to $(\ss,\sigma)$ by the construction above is non-split if and only if
$$[K(x,x)]\neq 1 \in \ssquare \qquad \forall x \in \ss^{\sigma}.$$
\end{thm}

\begin{proof}
We first prove the following lemma.
\begin{lm} \label{existenceHsplitOrthogonal}
Let $x$ be a non-zero element of a split three-dimensional simple Lie algebra $\ss$. Then, there exists $h$ in $\ss$ such that ${\rm ad}(h)$ is diagonalisable and which is orthogonal to $x$.
\end{lm}

\begin{proof}
Let $\lbrace h,e,f \rbrace$ be a standard $\sl(2,k)$-triple and $K$ be the Killing form of $\ss$. Since ${\rm Span}<e,f>$ is a hyperbolic plane, there exists $x'\in {\rm Span}<e,f>$ such that $K(x,x)=K(x',x')$. This implies that there exists $g \in SO(\ss)$ such that $g(x')=x$. Since $h$ is orthogonal to $x'$, $g(h)$ is orthogonal to $x$, we have $K(g(h),g(h))=K(h,h)$ and so the linear map ${\rm ad}(g(h))$ is diagonalisable.
\end{proof}

By Lemma \ref{existenceHsplitOrthogonal}, there exist $h,e,f \in \ss$ such that $\lbrace h,e,f \rbrace$ is a standard basis of $\ss\cong \sl(2,k)$ and $x \in {\rm Span}<e,f>$ where $x$ is a non-zero fixed point of $\sigma$. Since $\sigma$ is involutive and $\sigma(h)=-h$ we obtain that $\sigma$ is a reflection on the hyperbolic plane ${\rm Span}<e,f>$ and so there exists $a \in k^*$ such that $\sigma(e)=af$ and $\sigma(f)=\frac{1}{a}e$. The eigenspaces $\ll$ and $\ppp$ are 
$$\ll={\rm Span}<e+af>, \quad \ppp={\rm Span}<h,e-af>,$$
and so
$$\ss'=\ll \oplus \l \ppp={\rm Span}<e+af>\oplus~ {\rm Span}<\l h, \l (e-af)>.$$
We now calculate the structure constants of $\ss'$:
\begin{align*}
[\frac{\l h}{2},\frac{e+af}{2}]&=\frac{\l (e-af)}{2},\\
[\frac{e+af}{2},\frac{\l(e-af)}{2}]&=-a\frac{\l h}{2},\\
[\frac{\l(e-af)}{2},\frac{\l h}{2}]&=-\Lambda \frac{e+af}{2}.
\end{align*}
Since 
$$[\Lambda]=[\frac{K_{\ss}(x,x)}{2}]=[\frac{K_{\ss'}(e+af,e+af)}{2}]=[a] \in \square$$
it follows that $\ss'$ is isomorphic to $L(-\Lambda,-\Lambda)$. The quadratic form $<2\Lambda,2\Lambda,-2\Lambda^2>$ is isometric to $<-2,2\Lambda,2\Lambda>$ and so, by Proposition  \ref{equivalences quaternion alg}, the Lie algebra $L(-\Lambda,-\Lambda)$ is isomorphic to $L(-\Lambda,1)$.

\begin{lm} \label{L(D,D) split SSI somme Deux Carrés}
Let $\Delta,\Delta' \in k^*$. The Lie algebras $L(-\Delta,1)$ and $L(-\Delta',1)$ are isomorphic if and only if $[\Delta]=[\Delta']$ in $\ssquare$. In particular, $L(-\Delta,1)$ is split if and only if $\Delta$ is a sum of two squares.
\end{lm}

\begin{proof}
By Proposition \ref{equivalences quaternion alg}, $L(-\Delta,1)$ is isomorphic to $L(-\Delta',1)$ if and only $<1,-\Delta,-\Delta>$ is isometric to $<1,-\Delta',-\Delta'>$. By Witt's cancellation Theorem, $<1,-\Delta,-\Delta>$ is isometric to $<1,-\Delta',\Delta'>$ if and only if $<\Delta,\Delta>$ is isometric to $<\Delta',\Delta'>$. Since they have the same discriminant, they are isometric if and only if they represent a common element (Proposition 5.1 p.15 in \cite{LAM}), in other words if and only if $[\Delta]=[\Delta']\in\ssquare$.
\end{proof}
\end{proof}

This theorem motivates the following definition.

\begin{df}
Let $\ss$ be a split three-dimensional simple Lie algebra, let $K$ be its Killing form and let $\sigma$ be an automorphism of $\ss$. We say that $\sigma$ is of Cartan type if and only if:\\
$$\sigma\neq Id, \qquad \sigma^2=Id, \qquad [K(x,x)]\neq 1 \in \ssquare \quad \forall x \in \ss^{\sigma}.$$
Two automorphisms of Cartan type $\sigma$ and $\sigma'$ are said to be equivalent if
$$[K(x,x)]=[K(x',x')] \in \ssquare \qquad \forall x \in \ss^{\sigma}, ~ \forall x' \in \ss^{\sigma'}.$$
\end{df}

\begin{rem} \label{rem on split involutive automorphisms}
\begin{enumerate}[label=\alph*)]
\item The Killing form $K$ of a split three-dimensional simple Lie algebra $\ss$ represents all the elements of $k$. Hence for all $\alpha \in k^*$ such that $[\alpha]\neq 1\in \ssquare$, there exists an automorphism of Cartan type $\sigma$ of $\ss$ such that
$$K(x,x)=\alpha \qquad \forall x \in \ss^{\sigma}.$$
\item If $k=\mathbb{R}$, an automorphism $\sigma$ of $\sl(2,\mathbb{R})$ is of Cartan type if and only if $\sigma$ is a Cartan involution.
\item If $x\in \ss$ satisfies to $[K(x,x)]\neq 1\in \ssquare$ then $\frac{K(x,x)}{2}$ is not a square.
\end{enumerate}
\end{rem}

We now study the non-split three-dimensional simple Lie algebras which can be obtained by the construction above.

\begin{df}
A non-split three-dimensional simple Lie algebra $\ss'$ is said to be obtainable if there exists an automorphism of Cartan type $\sigma$ of a split three-dimensional simple Lie algebra $\ss$ such that $\ss'$ is isomorphic to the Lie algebra associated to $(\ss,\sigma)$ by the construction above.
\end{df}

We now summarise various conditions for a non-split three-dimensional simple Lie algebra to be obtainable in the following proposition.

\begin{pp} \label{pp obtainable}
Let $\ss'$ be a non-split three-dimensional simple Lie algebra and $K$ be its Killing form. The following are equivalent:
\begin{enumerate}[label=\alph*)]
\item\label{pp obtainable cond a} $\ss'$ is obtainable,
\item\label{pp obtainable cond b} there exist $x,h \in \ss'$ such that
$$h \bot x, \quad [K(x,x)]\neq 1 \in \ssquare \text{ and } [K(h,h)]=[K(x,x)] \in \square,$$
\item\label{pp obtainable cond c} $\ss'$ is isomorphic to $L(-\Delta,-\Delta)$ for some $\Delta \in k^*$,
\item\label{pp obtainable cond d} the Killing form of $\ss'$ represents $-2$.
\end{enumerate}
\end{pp}

\begin{proof}
Conditions \ref{pp obtainable cond a} and \ref{pp obtainable cond b} are equivalent by construction. As we saw in the proof of Theorem \ref{thm split pair -> non-split pair}, conditions \ref{pp obtainable cond a} and \ref{pp obtainable cond b} imply condition \ref{pp obtainable cond c}. Conversely, if $\ss'\cong L(-\Delta,-\Delta)$ for some $\Delta \in k^*$, we have $[\Delta]\neq 1\in \ssquare$ since $\ss'$ is non-split. By Remark \ref{rem on split involutive automorphisms}, $L(-\Delta,-\Delta)$ is obtainable and so Conditions \ref{pp obtainable cond a}, \ref{pp obtainable cond b} and \ref{pp obtainable cond c} are equivalent. We now show that Conditions \ref{pp obtainable cond a}, \ref{pp obtainable cond b} and \ref{pp obtainable cond c} are equivalent to Condition \ref{pp obtainable cond d}. If the Killing form $K$ of $\ss$ represents $-2$, there exists $\delta$ and $\gamma$ in $k^*$ such that $K$ is isometric to $<-2,\delta,\gamma>$. Since $disc(K)=[-2] \in \square$ and $disc(<-2,\delta,\gamma>)=[-2\delta\gamma]\in \square$ we have $[\gamma]=[\delta] \in \square$ and then $K$ is isometric to the quadratic form $<-2,\delta,\delta>$ which is isometric to the Killing form of $L(\frac{-\delta}{2},\frac{-\delta}{2})$. Hence $\ss$ is isomorphic to $L(\frac{-\delta}{2},\frac{-\delta}{2})$. Conversely, the Killing form of $L(\delta,\delta)$ is isometric to $<-2\delta,-2\delta,-2>$ and hence represents $-2$.
\end{proof}

Consider the Brauer group $B(k)$ of $k$. The elements of $B(k)$ are in $1:1$ correspondence with the isomorphism classes of central division algebras over $k$ (for details see Chap.IV of \cite{LAM}). Non-isomorphic quaternion algebras represent different elements of $B(k)$ and, we now consider the elements of $B(k)$ represented by the quaternion algebras constructed from obtainable non-split three-dimensional simple Lie algebras (see Section \ref{section generalities}):

\begin{df}
Let $B(k)$ be the Brauer group of $k$. Define
$$H(k):=[\mathcal{H}(\sl(2,k))]\cup\Big\{ [\mathcal{H}(\ss)]\in B(k) ~ | ~ \ss \text{ is an obtainable non-split three-dimensional simple Lie algebra} \Big\}.$$
\end{df}

A quaternion algebra is of order $2$ in the Brauer group but, in general, the set of classes of quaternion algebras in $B(k)$ is not a subgroup. However, the set $H(k)$ of classes of quaternion algebras associated to obtainable non-split three-dimensional simple Lie algebra $\ss$ is a subgroup of $B(k)$.

\begin{pp}\label{pp subgroup of the Brauer group}
The set $H(k)$ is a subgroup of the Brauer group $B(k)$ isomorphic to $\ssquare$.
\end{pp}

\begin{proof}
Let $\ss$ and $\ss'$ be obtainable non-split three-dimensional simple Lie algebras. By Proposition \ref{pp obtainable}, there exist $\Delta$ and $\Delta'$ in $k^*$ such that $\ss\cong L(-\Delta,1)$ and $\ss'\cong L(-\Delta',1)$ and by Proposition \ref{pp quaternion from 3DSLA} the quaternion algebras $\mathcal{H}(\ss)$ and $\mathcal{H}(\ss')$ are isomorphic respectively to $\bigl(\begin{smallmatrix}
-1, \Delta\\ \hline k
\end{smallmatrix} \bigr)$ and $\bigl(\begin{smallmatrix}
-1, \Delta'\\ \hline k
\end{smallmatrix} \bigr)$. By Linearity (see Theorem 2.11 p.60 in \cite{LAM}) we have
\begin{equation*}
\bigl(\begin{smallmatrix}
-1, \Delta\\ \hline k
\end{smallmatrix} \bigr)\otimes\bigl(\begin{smallmatrix}
-1, \Delta'\\ \hline k
\end{smallmatrix} \bigr)\cong \bigl(\begin{smallmatrix}
-1, \Delta\Delta'\\ \hline k
\end{smallmatrix} \bigr)\otimes M_2(k).
\end{equation*}
Hence, in the Brauer group we have
$$[\bigl(\begin{smallmatrix}
-1, \Delta\\ \hline k
\end{smallmatrix} \bigr)]\cdot[(\begin{smallmatrix}
-1, \Delta'\\ \hline k
\end{smallmatrix} \bigr)]=[(\begin{smallmatrix}
-1, \Delta\Delta'\\ \hline k
\end{smallmatrix} \bigr)]$$
and so $H(k)$ is a subgroup of $B(k)$. Furthermore the map $[\bigl(\begin{smallmatrix}
-1, \Delta\\ \hline k
\end{smallmatrix} \bigr)]\mapsto \Delta$ defines a group isomorphism between $H(k)$ and $\ssquare$.
\end{proof}

We can summarise the correspondences between automorphisms of Cartan type of $\sl(2,k)$, obtainable non-split three-dimensional simple Lie algebras, $\ssquare$ and $H(k)$ as follows.

\begin{thm}\label{thm corresponces Brauer}
We have the following correspondences
$$\begin{Bmatrix}
\text{equivalence classes of}\\
\text{automorphisms of}\\
\text{Cartan type of }\sl(2,k)
\end{Bmatrix}\Leftrightarrow\begin{Bmatrix}
\text{isomorphism classes of}\\
\text{obtainable non-split three-}\\
\text{dimensional simple Lie algebras}
\end{Bmatrix}\Leftrightarrow\begin{Bmatrix}
\text{elements}\\
\text{in}\\
\ssquare\end{Bmatrix}\Leftrightarrow\begin{Bmatrix}
\text{elements in the subgroup}\\
H(k)\text{ of the Brauer}\\
\text{group }B(k)\text{ of }k
\end{Bmatrix}.$$
\end{thm}

\section{Criteria to be obtainable over local and global fields} \label{section local and global fields}

In this section, after recalling definitions, we characterise which non-split three-dimensional simple Lie algebras are obtainable over local and global fields in terms of the Hilbert symbol and the Legendre symbol (see \cite{Vigneras80} for details about quaternions algebras over local and global fields). We also give examples of obtainable and unobtainable non-split three-dimensional simple Lie algebras over the field of rationals.

\begin{df}
Let $\alpha,\beta \in k^*$. We define the Hilbert symbol $(\alpha,\beta)\in \lbrace \pm 1 \rbrace$ as follows:
$$(\alpha,\beta):=\left\{
    \begin{array}{ll}
    	1 & \mbox{if the binary form } <\alpha,\beta> \mbox{ represents } 1, \\
        -1 & \mbox{otherwise.}
    \end{array}
\right.$$
\end{df}

\begin{pp} \label{pp L(a,b) split ssi (-a,-b)=1}
Let $\alpha,\beta \in k^*$. The Lie algebra $L(\alpha,\beta)$ is split if and only if $(-\alpha,-\beta)=1$.
\end{pp}

\begin{proof}
By Corollary \ref{cor 3DSLA iso iff Killing iso} the Lie algebra $L(\alpha,\beta)$ is split if and only $<-2\beta,-2\alpha,-2\alpha\beta>$ is isotropic. Since the norm form of the imaginary part of the quaternion algebra $\left(\frac{-\beta,-\alpha}{k}\right)$ is isometric to $<-(-\beta),-(-\alpha),(-\alpha)(-\beta)>$, then by Theorem 2.7 p.58 of \cite{LAM} we have that $L(\alpha,\beta)$ is split if and only if $(-\alpha,-\beta)=1$.
\end{proof}

We now introduce the Legendre symbol.

\begin{df}
For an odd prime $p$ and $a \in \mathbb{Z}$, the Legendre symbol is defined by:
$$\left(\frac{a}{p}\right):=\left\{
    \begin{array}{ll}
    	0 & \mbox{if p divides a,}  \\
        1 & \mbox{if a is a square modulo p,} \\
        -1 & \mbox{otherwise.}
    \end{array}
\right.$$
\end{df}

\begin{rem}
There is a formula for the Legendre symbol:
$$\left(\frac{a}{p}\right)=a^{\frac{p-1}{2}} \pmod p.$$
\end{rem}

If $k_{\PP}$ is a non-dyadic local field, we denote by $\overline{k_{\PP}}$ its residue class field and denote by $v_{\PP}$ its valuation. Recall that for any prime $p$, the fields $\mathbb{Q}_p$ and $\mathbb{F}_p((t))$ are examples of local fields whose residue class field is isomorphic to $\mathbb{F}_p$. The Hilbert symbol over a local field can be re-written in terms of the Legendre symbol as follows. Let $\alpha,\beta \in k_{\PP}^*$. We note $a=v_{\PP}(\alpha)$ and $b=v_{\PP}(\beta)$. By Corollary p.211 of \cite{Serre68} we have
$$(\alpha,\beta)=\Big((-1)^{ab}\frac{\alpha^b}{\beta^a}\Big)^{\frac{|\overline{k_{\PP}}|-1}{2}}.$$
In particular if $|\overline{k_{\PP}}|$ is prime, we have
\begin{equation}\label{equation hilbert in terms of legendre}
(\alpha,\beta)=\left(\frac{(-1)^{ab}\frac{\alpha^b}{\beta^a}}{|\overline{k_{\PP}}|}\right).
\end{equation}

By Proposition \ref{equivalences quaternion alg} and by Theorem 2.2 p.152 of \cite{LAM}, there is up to isomorphism only one non-split three-dimensional simple Lie algebra over $k_{\PP}$. The standard model is $L(-u,-\pi)$ where $u,\pi \in k_{\PP}$, $v_{\PP}(u)=0$, $\bar{u}\notin \overline{k_{\PP}}^{*2}$ and $v_{\PP}(\pi)=1$.
 
\begin{pp} \label{critere local field}
Let $k_{\PP}$ be a non-dyadic local field and $\overline{k_{\PP}}$ be its residue class field. A non-split three-dimensional simple Lie algebra over $k_{\PP}$ is obtainable if and only if $|\overline{k_{\PP}}| \equiv 3 \pmod 4$.
\end{pp}

\begin{proof}
We first need the following Lemma
\begin{lm} \label{lm -1 square local field}
We have $-1 \in k_{\PP}^{*2}$ if and only if $|\overline{k_{\PP}}| \equiv 1 \pmod 4$ if and only if $\quot{k_{\PP}^*}{{k_{\PP}^*}_{-1}} =\lbrace 1 \rbrace$.
\end{lm}

\begin{proof}
We have $-1 \in k_{\PP}^{*2}$ if and only if $|\overline{k_{\PP}}| \equiv 1 \pmod 4$ by Corollary 2.6 p.154 of \cite{LAM}. We have $\quot{k_{\PP}^*}{{k_{\PP}^*}_{-1}}  =\lbrace 1 \rbrace$ if and only if the quadratic form $<1,1,-\Delta>$ is isotropic for all $\Delta \in k_{\PP}^*$. The quadratic form $<1,1,-\Delta>$ is isotropic for all $\Delta \in k_{\PP}^*$ if and only if $(-1,\Delta)=1$ for all $\Delta \in k_{\PP}^*$. However, since $k_{\PP}$ is a local field, $(-1,\Delta)=1$ for all $\Delta \in k_{\PP}^*$ if and only if $-1$ is a square by Proposition 7 p.208 of \cite{Serre68}.
\end{proof}
Let $\ss$ be a non-split three-dimensional simple Lie algebra. If $|\overline{k_{\PP}}| \equiv 1 \pmod 4$, then by the previous lemma and the proposition \ref{pp obtainable} there is no automorphism of Cartan type of $\sl(2,k)$. If $|\overline{k_{\PP}}| \equiv 3 \pmod 4$, then by Lemma \ref{lm -1 square local field}, there exists an automorphism of Cartan type $\sigma$ of $\sl(2,k_{\PP})$. Let $\ss'$ be the non-split three-dimensional simple Lie algebra associated to $(\sl(2,k_{\PP}),\sigma)$ by the construction of Section \ref{Section 3DLA construction}. Since there is up to isomorphism one non-split three-dimensional simple Lie algebra over $k_{\PP}$, the Lie algebras $\ss$ and $\ss'$ are isomorphic and so $\ss$ is obtainable.
\end{proof}

Recall that the global fields are the number fields and the finite extensions of the function fields $\mathbb{F}_q(t)$. Using the Hasse-Minkowski theorem (\cite{LAM} p.170) and the previous proposition we obtain the following characterisation of obtainable non-split three-dimensional simple Lie algebras over global fields.

\begin{pp} \label{thm global field}
Let $k$ be a global field. A non-split three-dimensional simple Lie algebra $\ss$ is obtainable if and only if it satisfies the following conditions:
\begin{enumerate}[label=\alph*)]
\item over every non-archimedean completion $k_{\PP}$ of $k$ such that $|\overline{k_{\PP}}| \equiv 1 \pmod 4$, the Killing form of $\ss$ is isotropic,
\item the Killing form of $\ss$ represents $-2$ over all dyadic completions.
\end{enumerate}
\end{pp}

\begin{rem}
Condition $b)$ is automatically satisfied if $k$ is of characteristic not two and a finite extension of a function field.
\end{rem}

\begin{proof}
Using Proposition \ref{pp obtainable}, the Lie algebra $\ss$ is obtainable if and only if its Killing form $K$ represents $-2$. Moreover, $K$ represents $-2$ if and only if the quadratic form $K\perp<2>$ is isotropic. By the Hasse-Minkowski theorem (\cite{LAM} p.170) we know that $K\perp<2>$ is isotropic over $k$ if and only if $K\perp<2>$ is isotropic over every completion $k_{\PP}$ of $k$ (including the dyadic completions). We now show that this condition is automatically satisfied for archimedean completions and non-archimedean completions $k_{\PP}$ such that $|\overline{k_{\PP}}| \equiv 3 \pmod 4$.
\vspace{0.2cm}

If $|\cdot|_{\PP}$ is an archimedean absolute value on $k$ then $k_{\PP}$ is either $\mathbb{R}$ or $\mathbb{C}$. If $k_{\PP}=\mathbb{C}$, the quadratic form $K\perp<2>$ is isotropic and if $k_{\PP}=\mathbb{R}$, the signature of $K\perp <2>$ is indefinite and then isotropic.
\vspace{0.2cm}

Using Propositions \ref{pp obtainable} and \ref{critere local field} we have that over every non-archimedean completion $k_{\PP}$ of $k$ such that $|\overline{k_{\PP}}| \equiv 3 \pmod 4$ the quadratic form $K\perp <2>$ is isotropic. Finally, using again Proposition \ref{critere local field} we have that over every non-archimedean completion $k_{\PP}$ of $k$ such that $|\overline{k_{\PP}}| \equiv 1 \pmod 4$ the quadratic form $K\perp <2>$ is isotropic if and only if $K$ is isotropic. This complete the proof of the proposition.
\end{proof}

This result can be re-expressed as follows.

\begin{cor}
Let $k$ be a global field. Let $L(\alpha,\beta)$ be a non-split three-dimensional simple Lie algebra, where $\alpha,\beta \in k^*$. The Lie algebra $L(\alpha,\beta)$ is obtainable if and only if it satisfies the following conditions:
\begin{enumerate}[label=\alph*)]
\item over every non-archimedean non-dyadic completion $k_{\PP}$ of $k$ such that $|\overline{k_{\PP}}| \equiv 1 \pmod 4$ and such that $v_{\PP}(\alpha)$ or $v_{\PP}(\beta)$ is non-zero we have 
\begin{equation} \label{globalcondition}
(\alpha,\beta)_{k_{\PP}}=1
\end{equation}
where $v_{\PP}$ is the valuation associated to $k_{\PP}$.
\item the quadratic form $<\alpha,\beta,\alpha\beta,-1>$ is isotropic over all dyadic completions.
\end{enumerate}
\end{cor}

\begin{proof}
Let $K$ be the Killing form of $L(\alpha,\beta)$. By Remark \ref{pp L(a,b) split ssi (-a,-b)=1}, the quadratic form $K$ is isotropic over a non-archimedean completion $k_{\PP}$ of $k$ such that $|\overline{k_{\PP}}| \equiv 1 \pmod 4$ if and only if $(-\alpha,-\beta)_{k_{\PP}}=1$. But since $-1$ is a square in $k_{\PP}$ by Lemma \ref{lm -1 square local field} this is equivalent to have $(\alpha,\beta)_{k_{\PP}}=1$.
\end{proof}

Over $\mathbb{Q}$ this implies the following result:

\begin{pp} \label{ex general sur Q}
Let $L(\alpha,\beta)$ be a non-split three-dimensional simple Lie algebra over $\mathbb{Q}$, where $\alpha,\beta \in \mathbb{Q}^*$. The Lie algebra $L(\alpha,\beta)$ is obtainable if and only if for every prime $p\equiv 1 \pmod 4$ such that $v_p(\alpha)$ or $v_p(\beta)$ is non-zero we have
$$\left(\frac{\frac{\alpha^{v_p(\beta)}}{\beta^{v_p(\alpha)}}}{p}\right)=1 \pmod p$$
where $v_p$ is the p-adic valuation associated to the prime $p$.
\end{pp}

\begin{rem}
For fixed $\alpha,\beta \in \mathbb{Q}^*$, the number of primes $p$ such that $v_p(\alpha)$ or $v_p(\beta)$ is non-zero is finite.
\end{rem}

\begin{proof}
The dyadic completion of $\mathbb{Q}$ is $\mathbb{Q}_2$ and
$$disc(<\alpha,\beta,\alpha\beta,-1>_{\mathbb{Q}_2})=-1 \notin \mathbb{Q}_2^{2*} $$
by Corollary p.40 in \cite{Cassels78}. Hence, $<\alpha,\beta,\alpha\beta,-1>_{\mathbb{Q}_2}$ is isotropic by Lemma 2.6 p.59 of \cite{Cassels78} and so the Killing form of $L(\alpha,\beta)$ represents $-2$ over $\mathbb{Q}_2$. We know that the non-archimedean completions of $\mathbb{Q}$ are the p-adic fields $\mathbb{Q}_p$ and the residue class field of $\mathbb{Q}_p$ is isomorphic to $\mathbb{F}_p$. Then, $|\overline{\mathbb{Q}_p}| \equiv 1 \pmod 4$ if and only if $p\equiv 1 \pmod 4$. If $p\equiv 1 \pmod 4$, from Equation \eqref{equation hilbert in terms of legendre}, we have
$$(\alpha,\beta)_{k_{\PP}}=\left(\frac{\frac{\alpha^{v_p(\beta)}}{\beta^{v_p(\alpha)}}}{p}\right).$$
In particular, if $v_p(\alpha)=0$ and $v_p(\beta)=0$, then the condition \eqref{globalcondition} is automatically satisfied. 
\end{proof}

Here are some examples of obtainable and unobtainable non-split three-dimensional simple Lie algebras over the field of rationals using Proposition \ref{ex general sur Q}.

\begin{ex}\label{examples over Q} Suppose that $k=\mathbb{Q}$.
\begin{enumerate}[label=\alph*)]
\item If $\alpha,\beta >0$, then $(-\alpha,-\beta)=-1$ and so $L(\alpha,\beta)$ is non-split by Proposition \ref{pp L(a,b) split ssi (-a,-b)=1}. In particular, the Lie algebras $L(2,3)$, $L(2,5)$ and $L(3,25)$ are non-split. The Lie algebra $L(2,3)$ is obtainable since there is no prime $p \equiv 1 \pmod 4$ such that $v_p(2)$ or $v_p(3)$ is non-zero. The Lie algebra $L(2,5)$ is unobtainable since for the prime $p=5$, we have $v_5(2)=0$, $v_5(5)=1$ and
$$\left(\frac{\frac{2^{v_5(5)}}{5^{v_5(2)}}}{5}\right)=2^2=-1 \pmod 5.$$
The Lie algebra $L(3,25)$ is obtainable since for the prime $p=5$, we have $v_5(3)=0$, $v_5(25)=2$ and
$$\left(\frac{\frac{3^{v_5(25)}}{25^{v_5(3)}}}{5}\right)=9^2=1 \pmod 5.$$
\item Using Proposition \ref{equivalences quaternion alg} and Example 2.17 p.63 of \cite{LAM} we have that the Lie algebra $L(3,-5)$ is non-split. Since for the prime $p=5$, we have $v_5(3)=0$, $v_5(-5)=1$ and
$$\left(\frac{\frac{3^{v_5(-5)}}{(-5)^{v_5(3)}}}{5}\right)=3^2=-1 \pmod 5,$$
then $L(3,-5)$ is unobtainable.
\item Let $p$ be an odd prime. We know from Example 2.14 p.62 of \cite{LAM} and Proposition \ref{equivalences quaternion alg} that $L(1,-p)$ is non-split if and only if $p\equiv 3 \pmod 4$. If $p\equiv 3 \pmod 4$, then the non-split Lie algebra $L(1,-p)$ is obtainable since there is no prime $p' \equiv 1 \pmod 4$ such that $v_{p'}(1)$ or $v_{p'}(-p)$ is non-zero.
\item Let $p$ be an odd prime. We know from Example 2.15 p.62 of \cite{LAM} and Proposition \ref{equivalences quaternion alg} that $L(2,-p)$ is non-split if and only if $p\equiv 5 \pmod 8$ or $p\equiv 7 \pmod 8$. If $p\equiv 7 \pmod 8$, then the non-split Lie algebra $L(2,-p)$ is always obtainable since there is no prime $p' \equiv 1 \pmod 4$ such that $v_{p'}(2)$ or $v_{p'}(-p)$ is non-zero. If $p\equiv 5 \pmod 8$, the non-split Lie algebra $L(2,-p)$ is always unobtainable since for the prime $p$, we have $v_p(2)=0$, $v_p(-p)=1$ and
$$\left(\frac{2}{p}\right)=-1 \pmod p$$
by the Quadratic Reciprocity Law (see p.181 in \cite{LAM}). 
\end{enumerate}
\end{ex}

\footnotesize
\bibliographystyle{alpha}
\bibliography{involutionsofsl2}
\end{document}